\documentclass[11pt,a4paper]{article}

\usepackage{epsf,epsfig,amsfonts,amsgen,amsmath,amstext,amsbsy,amsopn,amsthm,cases,listings,color
}
\usepackage{ebezier,eepic}
\usepackage{color}
\usepackage{multirow}
\usepackage{epstopdf}
\usepackage{graphicx}
\usepackage{pgf,tikz}
\usepackage{mathrsfs}
\usepackage[marginal]{footmisc}
\usepackage{enumitem}
\usepackage[titletoc]{appendix}
\usepackage{booktabs}
\usepackage{url}

\usepackage{pgf,tikz,pgfplots}
\usepackage{subfigure}
\usepackage{mathrsfs}
\usetikzlibrary{arrows}

\allowdisplaybreaks[1]

\definecolor{uuuuuu}{rgb}{0.27,0.27,0.27}
\definecolor{sqsqsq}{rgb}{0.1255,0.1255,0.1255}

\newtheorem{dfn}{Definition} [section]
\newtheorem{thm}[dfn]{Theorem}
\newtheorem{lemma}[dfn]{Lemma}
\newtheorem{prop}[dfn]{Proposition}

\newtheorem{coro}[dfn]{Corollary}
\newtheorem{conj}[dfn]{Conjecture}
\newtheorem{claim}[dfn]{Claim}

\newenvironment{pf}{\noindent{\bf Proof}}

\def\qed{\hfill \rule{4pt}{7pt}}

\def\lc{\left\lceil}

\def\rc{\right\rceil}

\def\lf{\left\lfloor}

\def\rf{\right\rfloor}

\begin{document}
\title{\bf\Large Sparse halves in $K_4$-free graphs}

\date{\today}

\author{Xizhi Liu\thanks{Department of Mathematics, Statistics, and Computer Science, University of Illinois, Chicago, IL, 60607 USA. Email: xliu246@uic.edu.
Part of the work was done when this author was a college student at the University of Science and Technology of China.}
\and
Jie Ma\thanks{School of Mathematical Sciences,
University of Science and Technology of China,
Hefei, Anhui 230026, China.
Email: jiema@ustc.edu.cn. Research supported in part by NSFC grant 11622110.}
}
\maketitle
\begin{abstract}
A conjecture of Chung and Graham states that every $K_4$-free graph on $n$ vertices contains a vertex set of size $\lfloor n/2 \rfloor$
that spans at most $n^2/18$ edges.
We make the first step toward this conjecture by showing that it holds for all regular graphs.
\end{abstract}

\section{Introduction}
Given graphs $G$ and $H$, we say $G$ is {\it $H$-free} if $G$ does not contain $H$ as a subgraph.
The celebrated Tur\'{a}n theorem \cite{T41} states that for every $r\ge 3$, the maximum number of edges in
an $n$-vertex $K_{r}$-free graph is uniquely achieved by the Tur\'{a}n graph $T_{r-1}(n)$,
which is the complete $(r-1)$-partite graph on $n$ vertices such that the sizes of every two parts differ by at most one.
Generalizing Tur\'{a}n's theorem, Erd\H{o}s \cite{E76} initialized the study of the following problem:
Given a constant $0 \le \alpha \le 1$, what is the minimum valve $\beta = \beta(\alpha,r)$
such that every $n$-vertex $K_{r}$-free graph contains a vertex set of size $\lf \alpha n \rf$ which spans at most $\beta n^2$ edges?
This is often referred as the {\it local density problem}.

The case $\alpha = 1/2$ is of special intersect.
Erd\H{o}s \cite{E97} offered $\$$250 for the first solution to the following long-standing conjecture on triangle-free graphs.

\begin{conj}[Erd\H{o}s, \cite{E76}]\label{conj-erdos-sparse-halves-triangle}
Every triangle-free graph on $n$ vertices contains a vertex set of size $\lf n/2 \rf$ that spans at most $n^2/50$ edges.
\end{conj}

Both of the balanced blow-ups of the 5-cycle and the Petersen graph show that the bound $n^2/50$ would be best possible if this conjecture is true.
Despite extensive research \cite{K95,KS06,NY15,BGR19}, Conjecture \ref{conj-erdos-sparse-halves-triangle} is still open.

A similar question also has been asked for $K_4$-free graphs.
Chung and Graham \cite{CG90}, and Erd\H{o}s, Faudree, Rousseau and Schelp \cite{EFRS94} posted the following conjecture.

\begin{conj}[Chung et al. \cite{CG90}, Erd\H{o}s et al. \cite{EFRS94}]\label{conj-Chung-Graham-sparse-half-K4}
Every $K_4$-free graph on $n$ vertices contains a vertices set of size $\lf n/2 \rf$ that spans at most $n^2/18$ edges.
\end{conj}

The Tur\'{a}n graph $T_{3}(n)$ shows that
the bound $n^2/18$ in Conjecture \ref{conj-Chung-Graham-sparse-half-K4} would be best possible if it is true.
A closely related conjecture of Erd\H{o}s (see \cite{EFPS88}), which was proved by Sudakov \cite{S07},
states that every $K_4$-free graphs on $n$ vertices can be made bipartite by deleting at most $n^2/9$ edges.
An interesting interplay between these problems for {\it regular} graphs was observed by Krivelevich \cite{K95},
where he pointed out that a bound in the local density problem can imply a bound (doubled) in the problem of making a graph bipartite;
also see \cite{S07} for an illustration.

The main result of this paper is to confirm Conjecture \ref{conj-Chung-Graham-sparse-half-K4} for all {\it regular} graphs.
We prove it in the following form, which also characterizes the unique extremal graph.

\begin{thm}\label{thm-sparse-half-regular-K4}
Let $G$ be a $K_4$-free regular graph on $n$ vertices.
If every vertex set of size $\lf n/2 \rf$ in $G$ spans at least $n^2/18$ edges,
then $n$ is divisible by 6 and $G\cong T_{3}(n)$.
\end{thm}

We would like to remark that our proof of Theorem \ref{thm-sparse-half-regular-K4} actually shows that
Conjecture \ref{conj-Chung-Graham-sparse-half-K4} holds for all {\it almost regular} graphs,
i.e. graphs whose difference of maximum degree and minimum degree is bounded by $\epsilon n$
for some absolute constant $\epsilon>0$.\footnote{Our calculations indicate that $\epsilon$ can be chosen as $\epsilon=1/500$.}

As a corollary, Theorem \ref{thm-sparse-half-regular-K4} implies the following slightly stronger version of Sudakov's theorem in the case of regular graphs.

\begin{coro}\label{corollary-K4-bipartite}
Let $n \in \mathbb{N}$ be even.
Then every regular $K_4$-free graph on $n$ vertices can be made bipartite by removing at most $n^2/9$ edges such that each part has size exactly $n/2$.
\end{coro}

For odd $n\in \mathbb{N}$, one could easily obtain a similar result as in Corollary \ref{corollary-K4-bipartite}.

We now introduce a crucial tool in our proof of Theorem \ref{thm-sparse-half-regular-K4},
which also can be viewed as a strengthening of the local density problem.
Erd\H{o}s, Faudree, Rousseau and Schelp conjectured in \cite{EFRS94} that for every $\alpha \in [17/30,1]$, every triangle-free graph on $n$ vertices contains a vertex set of size $\lf \alpha n \rf$
that spans at most $(2\alpha-1)n^2/4$ edges.
This was confirmed by Krivelevich \cite{K95} for all $\alpha\in[3/5,1]$.
The coming result shows that the bound $(2\alpha-1)n^2/4$ can be improved in the range where $\alpha$ is relatively large.

\begin{thm}\label{thm-local-density-K3}
Let $\alpha, c\in [0,1]$ satisfy $\alpha+c\geq 1$.
Then the following hold:
\begin{itemize}
\item [(1).] Every $n$-vertex triangle-free graph with $cn^2$ edges contains a vertex set of size $\lf \alpha n \rf$ that spans
at most $(2\alpha -1) c n^2$ edges.
\item [(2).] Assume that $\alpha n \in \mathbb{N}$ and $G$ is an $n$-vertex triangle-free graph.
If every vertex set of size $\alpha n$ in $G$ spans at least $(2\alpha -1) c n^2$ edges, then $G$ is regular, and vice versa.
\end{itemize}
\end{thm}

Note that by Mantel's theorem \cite{M07}, we have $(2\alpha -1) c n^2 \le (2\alpha -1)n^2/4$.

The rest of the paper is organized as follows.
In Section 2 we present some preliminary results.
In Section 3 we prove Theorem~\ref{thm-local-density-K3}.
In Section 4 we complete the proof of Theorem~\ref{thm-sparse-half-regular-K4}, by dividing it into three parts according to the edge density.
In Section 5 we conclude this paper by mentioning some related problems.


\section{Preliminaries}
We first introduce our notation (which is conventional).
Given a graph $G$, we will use $V(G)$ and $E(G)$ to denote its vertex set and edge set, respectively.
Let $e(G) = |E(G)|$.
We use $d(G),\Delta(G),\delta(G)$ to denote the average degree, maximum degree, and minimum degree of $G$, respectively.
For $v\in V(G)$, let $N_G(v)$ be the set of the neighbors of $v$ in $G$ and let $d_{G}(v) = |N_{G}(v)|$.
For $S\subset V(G)$ we use $G[S]$ to express the induced subgraph of $G$ on $S$ and let $e_{G}(S)$ be the number of edges in $G[S]$.
For two disjoint vertex sets $S,T \subset V(G)$, let $G[S,T]$ be the induced bipartite subgraph of $G$ with two parts $S,T$ and
let $e_{G}(S,T)$ be the number of edges in $G[S,T]$.
If it is clear from the context we omit the subscript $G$.
We also omit floors and ceilings when they are not essential in our proofs.

The following propositions can be found in literatures (e.g. \cite{KS06}).
For completeness we include their proofs here.

\begin{prop}\label{prop-1}
Let $0 \le \alpha \le 1$.
Then every $n$-vertex graph $G$ with $e$ edges contains a vertex set of size $\alpha n$ that spans at most $\alpha^2 e$ edges.
\end{prop}
\begin{proof}
Choose $S \subset V(G)$ with $|S| = \alpha n$ uniformly at random.
Then for every edge $e$, the probability that $e$ is contained in $S$ is $\frac{\alpha n}{n} \cdot \frac{\alpha n-1}{n-1} \le \alpha^2$.
So, the expected value of $e(S)$ is at most $\alpha^2 e$.
Hence there exists a vertex set of size $\alpha n$ in $G$ that spans at most $\alpha^2e$ edges.
\end{proof}

\begin{prop}\label{prop-2}
Let $G$ be an $n$-vertex graph with $e$ edges.
Let $A\cup B = V(G)$ be a partition with $|A| = \alpha n \le n/2$.
Then there exists $S\subset B$ with $|S| = \left(1/2-\alpha\right)n$ such that
\begin{align}
e(A\cup S) & \le
e(A)+\frac{1/2-\alpha}{1-\alpha}e(A,B)+\left(\frac{1/2-\alpha}{1-\alpha}\right)^2e(B) \notag \\
& =e(G)-\frac{1}{2(1-\alpha)}e(A,B)-\frac{3/2-2\alpha}{2(1-\alpha)^2}e(B). \notag
\end{align}
\end{prop}
\begin{proof}
Choose $S \subset B$ with $|S| = \left(1/2-\alpha\right)n$ uniformly at random.
Then, for every $e\in E(G[A,B])$ the probability that $e$ is contained in $A\cup S$ is $\frac{1/2-\alpha}{1-\alpha}$.
Similar to the proof of Proposition \ref{prop-1},
for every $e'\in E(G[B])$ the probability that $e'$ is contained in $S$ is at most $\left(\frac{1/2-\alpha}{1-\alpha}\right)^2$.
So, the expected value of $e(A\cup S)$ is at most $e(A)+\frac{1/2-\alpha}{1-\alpha}e(A,B)+\left(\frac{1/2-\alpha}{1-\alpha}\right)^2e(B)$.
Therefore, there exists $S \subset B$ with $|S| = \left(1/2-\alpha\right)n$ such that the desired inequality holds.
\end{proof}

\section{Local densities in triangle-free graphs}
In this section we prove Theorem~\ref{thm-local-density-K3}.
First we show the following proposition for the ``vice versa" part of Theorem~\ref{thm-local-density-K3} (2).

\begin{prop}\label{prop-regular-triangle-free-graph}
Let $\alpha, c\in [0,1]$, $n\in \mathbb{N}$ such that $\alpha n \in \mathbb{N}$.
Suppose that $G$ is a triangle-free regular graph on $n$ vertices with $cn^2$ edges.
Then every $S\subseteq V(G)$ with $|S| = \alpha n$ spans at least $(2\alpha-1)cn^2$ edges.
\end{prop}
\begin{proof}
Let $S\subset V(G)$ be a set with size $\alpha n$ let $T=V(G)\setminus S$.
Since $G$ is regular, every vertex has degree $2cn$, which shows that
\begin{align}
2e(S)+e(S,T)= \sum_{v\in S}d(v) = 2\alpha cn^2 \notag
\quad {\rm and} \quad
e(S,T) \le \sum_{v\in T}d(v) = 2(1-\alpha)cn^2. \notag
\end{align}
Therefore,
\begin{align}
e(S)
= \frac{1}{2}\left(2e(S)+e(S,T) -e(S,T)\right)
\ge \frac{1}{2}(2\alpha cn^2-2(1-\alpha)cn^2) = (2\alpha-1)cn^2, \notag
\end{align}
which completes the proof of Proposition \ref{prop-regular-triangle-free-graph}.
\end{proof}

Now we prove Theorem \ref{thm-local-density-K3}.
The core of the proof is a probabilistic argument.
For convenience we will assume $\alpha n \in \mathbb{N}$ in the coming presentation,
while the proof for the case $\alpha n \not\in \mathbb{N}$ holds analogously.

\begin{proof}[Proof of Theorem \ref{thm-local-density-K3}]
Let $\alpha+c\geq 1$ and $G$ be an $n$-vertex triangle-free graph with $cn^2$ edges.
Our goal is to find a subset $S\subseteq V(G)$ with $|S| = \alpha n$ that spans at most $(2\alpha-1)cn^2$ edges.
It is clear that we may assume $\alpha<1$.
We divide the proof into two cases by considering the value of $\delta(G)$.

First suppose that $\delta(G)\ge(1-\alpha)n$.\footnote{We point out that this case holds even without requiring $\alpha+c\geq 1$.}
Suppose for the contrary that every subset of size $\alpha n$ spans more than $(2\alpha-1)cn^2$ edges.
For every $v\in V(G)$, let $B_v = N(v)$ and $A_v =V(G)\setminus B_v$.
Since $G$ is triangle-free, $B_v$ is an independent set and hence $e(A_v) + e(A_v,B_v) = c n^2$.
Let $d_v = d(v)/n$.
By the similar argument as in Proposition \ref{prop-2}, there exists $S\subseteq B_v$ with $|S|=(\alpha+d_v-1)n$ such that
\begin{align}
e(A_v\cup S) \le e(A_v)+\frac{\alpha+d_v-1}{d_v}e(A_v,B_v). \notag
\end{align}
Since $|A_v\cup S| = \alpha n$, by assumption, we have
\begin{align}
e(A_v)+\frac{\alpha+d_v-1}{d_v}e(A_v,B_v)
\ge e(A_v\cup S) > (2\alpha-1)c n^2, \notag
\end{align}
which together with $e(A_v) + e(A_v,B_v) = c n^2$ gives
\begin{align}
cn^2 - \frac{1-\alpha}{d_v}e(A_v,B_v)>(2\alpha-1)c n^2. \notag
\end{align}
Therefore,
\begin{align}
\sum_{v\in V(G)}\left( cn^2 - \frac{1-\alpha}{d_v}e(A_v,B_v) \right) d_v > \sum_{v\in V(G)} (2\alpha-1)cn^2 d_v, \notag
\end{align}
which implies
$$(1-\alpha)\sum_{v\in V(G)}e(A_v,B_v)< 2(1-\alpha)cn^2\sum_{v\in V(G)}d_{v}.$$
Since $\sum_{v\in V(G)}d_v= 2cn$ and $\alpha<1$,
this gives
\begin{align}
\sum_{v\in V(G)}e(A_v,B_v)<4c^2n^3. \notag
\end{align}
On the other hand, since $B_v$ is independent for each $v$, by the Cauchy-Schwarz inequality
\begin{align}
\sum_{v\in V(G)}e(A_v,B_v)= \sum_{v\in V(G)} \sum_{u\in N(v)} d(u)
= \sum_{u\in V(G)}\left(d(u)\right)^2 \notag \ge \frac{1}{n} \left(\sum_{u\in V(G)}d(u)\right)^2=4c^2 n^3, \notag
\end{align}
which is a contradiction.
Therefore, if $\delta(G)\ge(1-\alpha)n$,
then there exists a vertex set of size $\alpha n$ that spans at most $(2\alpha-1)cn^2$ edges.
Note that if every vertex set of size $\alpha n$ spans at least $(2\alpha-1)cn^2$ edges,
then by the above arguments, we see that $d(v)$ must be the same for all $v\in V(G)$, that is, $G$ is regular.

Now suppose that $\delta(G)<(1-\alpha)n$, where $\alpha+c\geq 1$.
Choose $v\in V(G)$ such that $d(v) = \delta(G) <(1-\alpha)n$ and remove $v$ from $G$.
We iteratively remove a vertex with the minimum degree in the remaining graph until
there is no vertex left or the remaining graph $G'$ satisfies $\delta(G') \ge (1-\alpha)n$.
Let $A$ denote the set of vertices we removed in this process and let $k=|A|/n$.
If $|A| = n$, then $e(G) < (1-\alpha)n^2 \le cn^2$, a contradiction.
So $|A| < n$, which implies that $G' \neq \emptyset$.
Since $\delta(G') \ge (1-\alpha)n$, we have $|V(G')| > (1-\alpha)n$.
Therefore, $k = |A|/n = (n-|V(G')|)/n < \alpha$.
Let $B = V(G)\setminus A$ and let $G' = G[B]$.
Also let $\tilde{n}=(1-k)n$ and $\tilde{\alpha}=\frac{\alpha-k}{1-k}$.
Since $\delta(G') \ge (1-\alpha) n = (1-\tilde{\alpha})\tilde{n}$,
by the previous case, there exists $S\subseteq B$ with $|S| = \tilde{\alpha}\tilde{n}$
such that $e(S) \le (2\tilde{\alpha}-1) e(B)$.
Now we obtain a desired subset $A\cup S$ in $G$ with size $|A\cup S|=kn+\tilde{\alpha}\tilde{n}=\alpha n$ and
\begin{align}
e(A\cup S)
& = e(A) + e(A,S) + e(S)
 \le   e(A) + e(A,B) + (2\tilde{\alpha}-1) e(B) \notag\\
& = (2\tilde{\alpha}-1) \left(e(A) + e(A,B)+e(B)\right) + 2(1-\tilde{\alpha})\left(e(A) + e(A,B)\right) \notag\\
& < (2\tilde{\alpha}-1) cn^2 + 2(1-\tilde{\alpha})k(1-\alpha)n^2 \notag \le (2\alpha -1)cn^2, \notag
\end{align}
where the second last inequality is strict since
$e(A) + e(A,B) < |A|(1-\alpha)n = k(1-\alpha)n^2$, and the last inequality follows from
\begin{align}
(2\tilde{\alpha}-1)c+2(1-\tilde{\alpha})(1-\alpha)k-(2\alpha-1)c
= \frac{2k(1-\alpha)(\alpha+c-1)}{k-1}
\le 0.   \notag
\end{align}
Therefore in case of $\delta(G)<(1-\alpha)n$, there always exists a subset of size $\alpha n$ spanning strictly less than $(2\alpha-1)cn^2$ edges.
Together with Proposition~\ref{prop-regular-triangle-free-graph},
we have finished the proofs of Theorem \ref{thm-local-density-K3} for both (1) and (2).
\end{proof}

\section{Sparse halves}
In this section we prove Theorem \ref{thm-sparse-half-regular-K4}.
Let $G$ be a $K_4$-free graph on $n$ vertices.
For a vertex set $S \subset V(G)$ with $|S| = \lf n/2 \rf$, we call it a {\it sparse half} of $G$ if $e(S) \le n^2/18$.

We will consider three cases regarding the edge density of $G$ and use quite different techniques in each case.
If $G$ is sparse, then we will use some probabilistic arguments to show that it contains a sparse half.
If $G$ is dense, then a result of Lyle \cite{L14} gives a nice structure on $G$ and this enables us to find a sparse half.
The most intricate case is when the edge density of $G$ is intermediate.
In this case, assuming $G$ does not contain a sparse half,
we will first find three large disjoint independent sets in $G$ (by using Theorem~\ref{thm-local-density-K3}),
and then building on these sets, use probabilistic arguments (in a complicated way) to derive a contradiction.
Finally, we infer Theorem \ref{thm-sparse-half-regular-K4} from these cases in Section~\ref{subsec:together}.

In the rest of this section we will state our results without assuming the parities of integers $n$.
However for convenience, in the proofs we will always view $n$ as even in order to avoid the floors
(while the same arguments also work for odd $n$).
For Theorem \ref{thm-sparse-half-regular-K4}, we will see in Section 4.4 that it suffices to only consider when $n$ is divisible by $6$.

\subsection{Sparse range}
In this section we will prove the following for graphs with few edges.

\begin{thm}\label{thm-case-sparse-K4}\label{thm-sparse-half-regular-K4-sparse}
Suppose that $G$ is a $K_4$-free graph on $n$ vertices with at most $0.26 n^2$ edges.
Then $G$ contains a vertex set of size $\lf n/2 \rf$ that spans strictly less than $n^2/18$ edges.
\end{thm}

We need the following two lemmas from \cite{S07} which are proved by probabilistic arguments. 
Let $t(G)$ denote the number of triangles in $G$.

\begin{lemma}[Sudakov, \cite{S07}]\label{lemma-max-cut-general-graph}
Every graph $G$ on $n$ vertices contains a bipartite subgraph $G'$ such that
\begin{align}
e(G') \ge \frac{1}{n}\sum_{v\in V(G)}\left(d(v)\right)^2 - \frac{2}{n}\sum_{v\in V(G)}e\left(N(v)\right)
\ge \frac{4\left(e(G)\right)^2}{n^2} - \frac{6t(G)}{n}. \notag
\end{align}
\end{lemma}

\begin{lemma}[Sudakov, \cite{S07}]\label{lemma-max-cut-K4}
Every $K_4$-free graph on $n$ vertices contains a bipartite subgraph $G'$ such that
\begin{align}
e(G') \ge \frac{e(G)}{2} + \frac{1}{n} \sum_{v\in V(G)}\left(\frac{4\left(e(N(v))\right)^2}{\left(d(v)\right)^2}-\frac{e(N(v))}{2}\right). \notag
\end{align}
\end{lemma}


The next lemma shows that if a $K_4$-free graph $G$ contains a large enough bipartite subgraph, then it contains a sparse half.

\begin{lemma}\label{lemma-upper-bound-max-cut-K4}
Let
$G$ be a $K_4$-free graph on $n$ vertices with $cn^2$ edges.
Suppose that there is a partition $A\cup B = V(G)$ such that $e(A,B) > 9c^2n^2/4$.
Then $G$ contains a vertex set of size $\lf n/2 \rf$ that spans strictly less than $n^2/18$ edges.
\end{lemma}
\begin{proof}
Suppose for the contrary that every vertex set of size $n/2$ in $G$ spans at least $n^2/18$ edges.
By Proposition \ref{prop-1}, we may assume that $c\geq \frac{2}{9}$.
Assume that $a := |A|/n \le 1/2$.
Applying Proposition \ref{prop-1} to $G[B]$, we obtain a vertex set $S\subset B$ with $|S| = n/2$
such that $e(S) \le \left(\frac{1/2}{1-a}\right)^2 e(B)$.
By assumption we have $\left(\frac{1/2}{1-a}\right)^2 e(B) \ge n^2/18$, which implies
\begin{align}
e(B) \ge \frac{2(1-a)^2}{9} n^2. \notag
\end{align}

Now applying Proposition \ref{prop-2} to $A\cup B$, we see that there exists $T \subset V(G)$ with $|T| = n/2$ such that
$A\subset T$ and $e(T) \le cn^2-\frac{1}{2(1-a)}e(A,B)-\frac{3/2-2a}{2(1-a)^2}e(B)$.
By assumption, we have
\begin{align}
\frac{n^2}{18}
\le cn^2-\frac{1}{2(1-a)}e(A,B)-\frac{3/2-2a}{2(1-a)^2}e(B)  \notag\le cn^2-\frac{1}{2(1-a)}e(A,B) - \frac{3/2-2a}{9} n^2, \notag
\end{align}
which implies
\begin{align}
e(A,B) \le 2(1-a) \left(c-\frac{3/2-2a}{9}-\frac{1}{18}\right)n^2
= \left(2(1-a)c-\frac{4}{9}(1-a)^2\right)n^2. \notag
\end{align}
Since the maximum of $2(1-a)c-4(1-a)^2/9$ is attained when $a = 1-9c/4$ (note that $a = 1-9c/4\leq 1/2$ as $c\geq 2/9$), we obtain
\begin{align}
e(A,B)
\le \left(2\left(1-\left(1-\frac{9c}{4}\right)\right)c-\frac{4}{9}\left(1-\left(1-\frac{9c}{4}\right)\right)^2\right)n^2
= \frac{9}{4}c^2n^2, \notag
\end{align}
a contradiction.
\end{proof}

Now we are ready to prove Theorem \ref{thm-sparse-half-regular-K4-sparse}.

\begin{proof}[Proof of Theorem \ref{thm-sparse-half-regular-K4-sparse}]
Let $c = e(G)/n^2$ and let $\lambda = 8/13$.
By Lemmas \ref{lemma-max-cut-general-graph} and \ref{lemma-max-cut-K4}, there exists a partition $A \cup B = V(G)$
such that
\begin{align}
e(A,B)
& \ge \left(1-\lambda\right)\left(\frac{1}{n}\sum_{v\in V(G)}\left(d(v)\right)^2 - \frac{2}{n}\sum_{v\in V(G)}e\left(N(v)\right)\right) \notag\\
& \quad +\lambda\left(\frac{e(G)}{2} + \frac{1}{n} \sum_{v\in V(G)}\left(\frac{4\left(e(N(v))\right)^2}{\left(d(v)\right)^2}-\frac{e(N(v))}{2}\right)\right) \notag\\
& = \frac{\lambda}{2}e(G) + \frac{4\lambda}{n}\sum_{v\in V(G)}\left(d(v)\right)^2\left(\left(\frac{e\left(N(v)\right)}{\left(d(v)\right)^2}\right)^2
  -\frac{2-3\lambda/2}{4\lambda}\frac{e\left(N(v)\right)}{\left(d(v)\right)^2}+\frac{1-\lambda}{4\lambda}\right). \notag
\end{align}
Since
\begin{align}
x^2 - \frac{2-3\lambda/2}{4\lambda}x + \frac{1-\lambda}{4\lambda}
\ge \frac{88\lambda-73\lambda^2-16}{256\lambda^2}, \notag
\end{align}
we obtain
\begin{align}
e(A,B)
& \ge \frac{\lambda}{2}e(G)+\frac{88\lambda-73\lambda^2-16}{64\lambda} \sum_{v\in V(G)}\frac{d(v)^2}{n} \notag \\
& \ge \frac{\lambda}{2}e(G)+\frac{88\lambda-73\lambda^2-16}{64\lambda}\sum_{v\in V(G)} \left(\frac{\sum_{v\in V(G)}d(v)}{n}\right)^2 \notag\\
& = \left(\frac{\lambda}{2} c +  \frac{88\lambda-73\lambda^2-16}{16\lambda} c^2\right) n^2
= \left(\frac{4}{13}c+\frac{111}{104}c^2\right)n^2. \notag
\end{align}
Since $\frac{4}{13}c+\frac{111}{104}c^2 > \frac{9}{4}c^2$ holds for all $c \in (0,\frac{32}{123})$ and $\frac{32}{123}>0.26$,
we derive that $e(A,B) > \frac{9}{4}c^2 n^2$ whenever $c\leq 0.26$.
Therefore, by Lemma \ref{lemma-upper-bound-max-cut-K4},
$G$ contains a vertex set of size $n/2$ that spans strictly less than $n^2/18$ edges.
\end{proof}

\subsection{Dense range}
In this section we prove the following for graphs with high minimum degree.

\begin{thm}\label{thm-case-dense-K4}
Suppose that $G$ is a $K_4$-free graph on $n$ vertices with $\delta(G) \ge 0.59n$.
Then $G$ contains a vertex set of size $\lf n/2 \rf$ that spans at most $n^2/18$ edges.
Moreover, if every vertex set of size $\lf n/2 \rf$ in $G$ spans at least $n^2/18$ edges,
then $G\cong T_{3}(n)$.
\end{thm}

To show this, we need a structural result on dense $K_4$-free graphs.
A $K_{r}$-free graph $G$ is {\it maximal} if adding any new edge to $G$ will result in a copy of $K_r$.
Let $G_1$ and $G_2$ be two vertex disjoint graphs.
The {\it join} of $G_1$ and $G_2$, denoted by $G_1 \vee G_2$, is a graph with $V(G_1 \vee G_2) = V(G_1) \cup V(G_2)$
and
\begin{align}
E(G_1 \vee G_2) & = E(G_1) \cup E(G_2) \cup \{uv: u \in V(G_1), v\in V(G_2)\}.\notag
\end{align}

\begin{thm}[Lyle, \cite{L14}]\label{thm-Lyle-structure-K4}
Let $G$ be a maximal $K_4$-free on $n$ vertices with $\delta(G)\ge 4n/7$.
Then either $G$ contains an independent set of size at least $4\delta(G)-2n$
or $G$ is the join of an independent set and a triangle-free graph.
\end{thm}

Our next lemma shows that if a $K_4$-free graph $G$ contains a large induced triangle-free graph,
then $G$ contains a sparse half.

\begin{lemma}\label{lemma-sparse-half-K4-with-large-K3}
Let $G$ be a $K_4$-free graph on $n$ vertices. Suppose that $G$ contains an induced triangle-free subgraph $\Gamma$ with at least $2n/3$ vertices.
Then $G$ contains a vertex set of size $n/2$ that spans at most $n^2/18$ edges.
Moreover, if $|V(\Gamma)|>2n/3$,
then $G$ contains a vertex set of size $n/2$ which spans strictly less than $n^2/18$ edges.
\end{lemma}
\begin{proof}
Let $A \subset V(G)$ such that $\Gamma = G[A]$ and let $x = |A|/n$.
We may assume that $x \le 5/6$ since otherwise we could choose $A'\subset A$ with $|A'| = 5n/6$
and consider $G[A']$ instead.
Let $\alpha = 1/(2x)$. Then $\alpha \ge 3/5$.
By a result of Krivelevich on triangle-free graphs \cite{K95},
there exists $T \subset A$ with $|T| =\alpha |A|= n/2$ such that
\begin{align}
e(T)
\le \frac{2\times \frac{1}{2x}-1}{4}|A|^2
= \frac{(1-x)x}{4}n^2
\le  \frac{n^2}{18}, \notag
\end{align}
where in the last inequality we used the assumption that $x \ge 2/3$.
Notice that if $x>2/3$, then the inequality above is strict.
This proves the lemma.
\end{proof}

We also need the following slightly stronger version of Krivelevich's theorem on local densities of triangle-free graphs.
A proof is included in the appendix, which follows from a detailed analysis of Krivelevich's proof in \cite{K95}
as well as the proof of Erd\H{o}s et al. in \cite{EFRS94}.

\begin{thm}[Krivelevich, \cite{K95}] \label{obsevation-triangle-free-Kriveleich}
Let $3/5< \alpha \le 1$, $n\in \mathbb{N}$ and $\alpha n \in \mathbb{N}$.
Let $G$ be a triangle-free graph on $n$ vertices.
If every vertex set of size $\alpha n$ in $G$ spans at least $\frac{2\alpha -1}{4}n^2$ edges,
then $G \cong T_2(n)$.
\end{thm}

Now we are ready to prove Theorem \ref{thm-case-dense-K4}.

\begin{proof}[Proof of Theorem \ref{thm-case-dense-K4}]
It is clear that to prove Theorem \ref{thm-case-dense-K4}, it suffices to consider maximal $K_4$-free graphs.
Let $G$ be a maximal $K_4$-free graph on $n$ vertices with $\delta(G) \ge 0.59n > 4n/7$.
Then by Theorem \ref{thm-Lyle-structure-K4}, either $G$ is the join of an independent set and a triangle-free graph
or $G$ contains an independent set of size at least $4\delta(G)-2n$.

First, suppose that the former case occurs, that is, $G$ is the join of an independent set $I$ and a triangle-free graph $\Gamma$.
Let $\alpha = |V(\Gamma)|/n$. So $|I| = (1-\alpha)n$.
We may assume that $\alpha > 1/2$ since otherwise we can simply choose a subset of $I$ with size $n/2$ which spans none of edges.
On the other hand, if $\alpha > 2/3$, then by Lemma \ref{lemma-sparse-half-K4-with-large-K3}, we are done.
So we may assume that $1/2< \alpha \le 2/3$.

Let $c = e(\Gamma)/(\alpha n)^2$.
If $c < 2/9$, then by Proposition \ref{prop-1}, there exists $S \subset V(\Gamma)$ with $|S| = n/2$ such that
\begin{align}
e(S) \le \left(\frac{1/2}{\alpha}\right)^2 c(\alpha n)^2 = \frac{1}{4} cn^2 < \frac{n^2}{18}. \notag
\end{align}
So we may assume that $c \ge 2/9$.
Since $\Gamma$ has at least $2\alpha^2n^2/9$ edges,
there exists some $v \in V(\Gamma)$ such that $d_{\Gamma}(v) \ge 4\alpha^2 n/9 \ge (\alpha - 1/2)n$,
where the last inequality holds as $1/2< \alpha \le 2/3$.
Let $T \subset N_{\Gamma}(v)$ be any subset with $|T| = (\alpha - 1/2)n$.
Since $\Gamma$ is triangle-free, $T$ is an independent set.
Therefore, $I\cup T$ has size $n/2$ and satisfies
\begin{align}
e(I\cup T) \le \left(1-\alpha\right)\left(\alpha-\frac{1}{2}\right)n^2 \le \frac{n^2}{18}, \notag
\end{align}
where the last inequality uses the assumption that $\alpha \le 2/3$.
Notice that if $\alpha < 2/3$, then the inequality above is strict.

Now we may assume that $G$ contains an independent set $A$ whose size is at least $4\delta(G)-2n \ge 9n/25$.
We may just take $A$ such that $|A| = 9n/25$. Let $B = V(G)\setminus A$.
By Proposition~\ref{prop-2}, there exists $U\subset B$ with $|U| = 7n/50$ such that
\begin{align}
e(A\cup U) &\le \frac{7/50}{16/25}e(A,B)+\left(\frac{7/50}{16/25}\right)^2 e(B)
= \frac{175}{1024}e(A,B) + \frac{49}{1024}(e(A,B)+e(B)) \notag\\
& \le \frac{175}{1024}\left(\frac{9}{25}n\times \frac{16}{25}n\right) + \frac{49}{1024} \times \frac{n^2}{3}
= \frac{4249}{76800} n^2
< \frac{n^2}{18}. \notag
\end{align}
Therefore, $A\cup U$ is a sparse half with $e(A\cup U)<n^2/18$.

From the arguments above, one could see that
if every vertex set of size $n/2$ in $G$ spans at least $n^2/18$ edges,
then $G$ must be the join of a triangle-free graph $\Gamma$ and an independent set $I$ with $|V(\Gamma)|=2n/3$.
Since every vertex set of size $n/2=\frac34|V(\Gamma)|$ in $\Gamma$ spans at least $n^2/18 = (2\cdot 3/4 - 1)|V(\Gamma)|^2/4$ edges,
by Theorem~\ref{obsevation-triangle-free-Kriveleich} we have $\Gamma \cong T_{2}(2n/3)$, which implies $G\cong T_{3}(n)$.
This finishes the proof of Theorem \ref{thm-case-dense-K4}.
\end{proof}

\subsection{Intermediate range}
In this section we will prove the following result for regular graphs.

\begin{thm}\label{thm-case-medium-K4}
Every $K_4$-free regular graph $G$ on $n$ vertices with $e(G)\le 0.297n^2$
contains a vertex set of size $\lf n/2 \rf$ that spans strictly less than $n^2/18$ edges.
\end{thm}

We would like to remind the reader that the assumption that $G$ is regular in Theorem \ref{thm-case-medium-K4}
can be replaced by $\Delta(G)-\delta(G) \le \epsilon n$ for some absolute (but small) constant $\epsilon>0$.
However, in order to keep the proof simple we shall only consider regular graphs.

The proof ideas are as follows.
First, under the assumption that all subsets of size $n/2$ span at least $n^2/18$ edges,
we show that $G$ must contain many triangles.
Then we show that there exists a partition $V(G) = V_1 \cup V_2 \cup V_3 \cup V_4$ such that $V_1,V_2,V_3$ are independent sets
and $|V_1|+|V_2|+|V_2|$ is relatively large.
Finally, utilizing this partition, we employ some ad hoc probabilistic arguments to find a sparse half and thus reach a contradiction.

Recall that $t(G)$ denotes the number of triangles in $G$.

\begin{lemma}\label{lemma-many-K3}
Let
$G$ be an $n$-vertex $K_4$-free graph with $cn^2$ edges and $n/2 \le \delta(G) \le \Delta(G) \le 9n/14$.
Suppose that every vertex set of size $\lf n/2 \rf$ in $G$ spans at least $n^2/18$ edges.
Then we have $t(G) \ge \frac{c}{27(1-2c)}n^3.$
\end{lemma}
\begin{proof}
For every $v\in V(G)$ let $\alpha_{v} = \frac{n}{2d(v)}$ and $c_{v} = e(N(v))/\left(d(v)\right)^2$.
First notice that $c_v \ge 2/9$ for all $v\in V(G)$,
since otherwise by Proposition \ref{prop-1}, there would be a set $S \subset N(v)$
with $|S| = n/2$ such that $e(S)\le \alpha_v^2 \cdot e(N(v)) < n^2/18$, a contradiction.

Fix $v \in V(G)$.
Since $d(v)\le 9n/14$, we see $\alpha_v \ge 7/9 \ge 1-c_v$.
By Theorem \ref{thm-local-density-K3} and our assumption,
there exists a vertex set $T\subset N(v)$ with $|T|=n/2$ such that $\frac{n^2}{18}\le e(T)\le (2\alpha_v-1) e(N(v))$.
This implies that for all $v\in V(G)$,
\begin{align}
e(N(v)) \ge \frac{n^2}{18}\frac{1}{2\alpha_v-1} = \frac{n^2}{18} \frac{d(v)}{n-d(v)}. \notag
\end{align}
Summing over all $v \in V(G)$, we obtain that $t(G) = \frac{1}{3}\sum_{v\in V(G)}e(N(v))$ is at least
\begin{align}
\frac{n^2}{54} \sum_{v\in V(G)}\frac{d(v)}{n-d(v)}\ge \frac{n^2}{54} \frac{ \sum_{v\in V(G)} d(v)}{ \sum_{v\in V(G)}(n-d(v))}= \frac{n^2}{54} \frac{2e(G)}{n^2-2e(G)}
= \frac{c}{27(1-2c)}n^3. \notag
\end{align}
Here we used Jensen's inequality and the fact that $\frac{x}{n-x}$ is concave up for $x \in (0,n)$.
\end{proof}

We also need the following lemma in \cite{S07}.
For distinct $u,v\in V(G)$,
let $N(uv)$ denote the set of common neighbors of $u$ and $v$ and let $d(uv) = |N(uv)|$.

\begin{lemma}[Sudakov, \cite{S07}]\label{lemma-three-independent-sets}
Every graph $G$ with $e$ edges and $m$ triangles contains a triangle $uvw$ such that
$d(uv)+d(vw)+d(wu) \ge \frac{9m}{e}.$
\end{lemma}

Notice that if $G$ is $K_4$-free, then $N(uv)$ is independent for all $uv\in E(G)$ and $N(uv) \cap N(vw) = \emptyset$ for all triangles $uvw$ in $G$.
The following lemma is an immediate consequence of Lemmas \ref{lemma-many-K3} and \ref{lemma-three-independent-sets}.

\begin{lemma}\label{lemma-lower-bound-size-three-independent}
Let
$G$ be an $n$-vertex $K_4$-free graph with $cn^2$ edges and $n/2 \le \delta(G) \le \Delta(G) \le 9n/14$.
Suppose every vertex set of size $\lf n/2 \rf$ in $G$ spans at least $n^2/18$ edges.
Then there exist three disjoint independent sets $V_1,V_2,V_3$ in $G$
such that
\begin{align}
|V_1| + |V_2| + |V_3| \ge \frac{n}{3(1-2c)}. \notag
\end{align}
\end{lemma}

Now we are ready to prove Theorem \ref{thm-case-medium-K4}.

\begin{proof}[Proof of Theorem \ref{thm-case-medium-K4}]
Let
$G$ be a $K_4$-free regular graph on $n$ vertices with $cn^2$ edges, where $c\le 0.297$.
Suppose that every vertex set of size $n/2$ in $G$ spans at least $n^2/18$ edges.
By Theorem~\ref{thm-sparse-half-regular-K4-sparse}, we may assume that $c\in [1/4, 0.297]$.
So every vertex has degree $2cn$ with $n/2\le 2cn\le 0.594n< 9n/14$.
Then by Lemma \ref{lemma-lower-bound-size-three-independent}, there exist three disjoint independent sets $V_1,V_2,V_3$ in $G$
such that
\begin{align}
|V_1| + |V_2| + |V_3| = g(c)n, \mbox{ ~where~ } g(c) = \frac{1}{3(1-2c)}. \notag
\end{align}
Let $V_4 = V(G) \setminus \left(\bigcup_{i=1}^{3}V_i\right)$ and let $x_i = |V_i|/n$ for $i \in [4]$.
Without loss of generality we may assume that $1/2> x_1 \ge x_2 \ge x_3$.
Let $e_{ij} = e(V_{i},V_{j})$ for all $\{i,j\} \subset [4]$ (so $e_{ij} = e_{ji}$) and let $e_4 = e(V_4)$.
We will consider four cases depending on the values of $x_1,x_2$ and $x_3$.


\textbf{Case 1:} $x_1 + x_2 \ge x_1 + x_3 \ge x_2 + x_3 \ge \frac{1}{2}$.

\begin{table}[htbp]
\centering
\begin{tabular}{llll}
\toprule
$V_1$                      & $V_2$                     &$V_3$                      & $V_4$\\
\midrule
$x_1$                      & $1/2-x_1$         & 0                         & 0 \\
$x_1$                      &0                          & $1/2-x_1$         & 0 \\
0                          & $x_2$                     & $1/2-x_2$         & 0 \\
$1/2-x_2$                          & $x_2$                     & 0         & 0\\
$1/2-x_3$                          & 0                     & $x_3$         & 0\\
 0                          & $1/2-x_3$                     & $x_3$         & 0\\
\bottomrule
\end{tabular}
\caption{\label{table.label} different schemes for choosing $n/2$ vertices from $G$.}
\end{table}

Now we choose different $n/2$ vertices from $G$ according to Table 1.
For example, the second row in Table 1 means to choose all vertices of $V_1$
and choose a set $S \subset V_2$ with $|S| = \left(1/2-x_1\right)n$ uniformly at random.
Then the expected value of $e(V_1 \cup S)$ is $\frac{1/2-x_2}{x_2}e_{12}$.
So there exists $S \subset V_2$ with $|S| = \left(1/2-x_1\right)n$ such that
$e(V_1 \cup S) \le \frac{1/2-x_2}{x_2}e_{12}$.
By assumption, we have
\begin{align}
\frac{1/2-x_1}{x_2}e_{12} \ge \frac{n^2}{18} \quad \Rightarrow \quad e_{12} \ge \frac{1}{36}\frac{x_2}{1-2x_1}n^2. \notag
\end{align}
Similarly, one can get from Table 1 that for all $(i,j) \in [3]\times [3]$ with $i\neq j$,
\begin{align}
\frac{1/2-x_i}{x_j}e_{ij} \ge \frac{n^2}{18} \quad \Rightarrow \quad e_{ij} \ge \frac{1}{36}\frac{x_j}{1-2x_i}n^2. \notag
\end{align}
Adding them up, we obtain that $e_{12}+e_{13}+e_{23} = \frac{1}{2}\sum_{i\neq j}e_{ij}$ is at least
\begin{align}
&\frac{1}{18}\left(\frac{x_2+x_3}{1-2x_1}+\frac{x_1+x_3}{1-2x_2}+\frac{x_2+x_1}{1-2x_3}\right)n^2 \notag =\frac{1}{18}\left(\frac{g(c)-x_1}{1-2x_1}+\frac{g(c)-x_2}{1-2x_2}+\frac{g(c)-x_3}{1-2x_3}\right)n^2. \notag
\end{align}
Since $\frac{g(c)-x}{1-2x}$ is concave up, by Jensen's inequality we see that
\begin{align}
e_{12}+e_{13}+e_{23} \ge \frac{1}{6}\cdot\frac{g(c)-(x_1+x_2+x_3)/{3}}{1-2(x_1+x_2+x_3)/{3}}n^2=\frac{g(c)}{3(3-2g(c))}n^2. \tag{1.1}
\end{align}

On the other hand, since $G$ is regular,\footnote{We point out that throughout the proof of Theorem \ref{thm-case-medium-K4}, this is the only place where we need the restriction that $G$ is regular.}
we have
\begin{align}
e_{14} + e_{24} + e_{34} + 2e_{4} = \sum_{v \in V_4}d(v) = 2cn \times |V_4| = 2c(1-g(c))n^2. \tag{1.2}
\end{align}
Since $G[V_4]$ is $K_4$-free, by Tur\'an's Theorem we get
\begin{align}
e_{4} \le \frac{1}{3}|V_{4}|^2 = \frac{(1-g(c))^2}{3}n^2. \tag{1.3}
\end{align}
Therefore, it follows from $(1.1),(1.2)$ and $(1.3)$ that (recall that $V_1, V_2, V_3$ are independent)
\begin{align}
cn^2 + \frac{(1-g(c))^2}{3}n^2 \ge e(G) + e_{4} \ge \frac{g(c)}{3(3-2g(c))}n^2 + 2c(1-g(c))n^2, \notag
\end{align}
which is a contradiction because
\begin{align}
h(c):=\frac{g(c)}{3(3-2g(c))}+2c\left(1-g(c)\right)-\left(c+\frac{(1-g(c))^2}{3}\right) \notag
\end{align}
is decreasing in $c$ for $c \in [1/4,0.297]$ and $h(0.297)>0$ (see \cite{CAL}). This proves Case 1.

\bigskip

\textbf{Case 2:} $x_2 + x_3 \le x_1 + x_3 \le x_1 + x_2 < \frac{1}{2}$.

Note that this case can exist only when $g(c) < 3/4$, which implies $c < 5/18$.

\begin{table}[htbp]
\centering
\begin{tabular}{llll}
\toprule
$V_1$                      & $V_2$                     &$V_3$                      & $V_4$\\
\midrule
$x_1$                      & $x_2$                     & $1/2-x_1-x_2$             & 0 \\
$x_1$                      & $1/2-x_1-x_3$             & $x_3$                     & 0 \\
$1/2-x_2-x_3$              & $x_2$                     & $x_3$                     & 0 \\
$x_1$                      &   0                       & 0                         &  $1/2-x_1$\\
0                          &  $x_2$                    & 0                         & $1/2-x_2$ \\
0                          &  0                        & $x_3$                     & $1/2-x_3$ \\
\bottomrule
\end{tabular}
\caption{\label{table.label} different schemes for choosing $n/2$ vertices from $G$.}
\end{table}

Now we choose $n/2$ vertices according to Table 2.
Then similar to Case 1, we obtain that
for every $k\in [3]$ and $\{i,j\}=[3]\backslash \{k\}$,
\begin{align}
e_{ij}+\frac{1/2-x_i-x_j}{x_k}(e_{ik}+e_{jk}) & \ge \frac{n^2}{18}, \tag{2.1}
\end{align}
and for all $i\in [3]$
\begin{align}
\frac{1/2-x_i}{x_4} e_{i4} + \left( \frac{1/2-x_i}{x_4} \right)^2 e_{4} & \ge \frac{n^2}{18}. \tag{2.2}
\end{align}
By simplifying the linear combination of
\begin{align}
\sum_{k\in [3]}\left(\frac{x_k}{x_4} \times (2.1)\right) + \sum_{i\in [3]}\left( \frac{x_4}{1/2-x_i} \times (2.2)\right), \notag
\end{align}
we can derive that
\begin{align}
& e(G)+ \frac{e_4}{2x_{4}} \ge \frac{x_1+x_2+x_3}{18x_4}n^2 + \frac{x_4}{18}\left(\frac{1}{1/2-x_1}
 +\frac{1}{1/2-x_2}+\frac{1}{1/2-x_3}\right)n^2 \notag\\
\ge & \frac{x_1+x_2+x_3}{18x_4}n^2 + \frac{x_4}{18} \times \frac{3}{1/2-(x_1+x_2+x_3)/{3}}n^2 = \frac{1-x_4}{18x_4}n^2+\frac{x_4}{3-2(1-x_4)}n^2. \notag
\end{align}
Since $e_4 \le |V_{4}|^2/3 = x_{4}^2n^2/3$ and $x_4 = 1-g(c)$, the inequality above implies
\begin{align}
c+ \frac{1-g(c)}{6}\ge \frac{1-(1-g(c))}{18(1-g(c))}+\frac{1-g(c)}{3-2(1-(1-g(c)))}= \frac{g(c)}{18(1-g(c))}+\frac{1-g(c)}{3-2g(c)}, \notag
\end{align}
which is a contradiction because
\begin{align}
k(c):= c+ \frac{1-g(c)}{6} - \left(\frac{g(c)}{18(1-g(c))}+\frac{1-g(c)}{3-2g(c)} \right) \notag
\end{align}
is strictly smaller than $0$ for $c \in [1/4,5/18)$ (see \cite{CAL}). This proves Case 2.

\bigskip

\textbf{Case 3:} $x_2 + x_3< \frac{1}{2}\le x_1 + x_3 \le x_1 + x_2$.

\begin{table}[htbp]
\centering
\begin{tabular}{llll}
\toprule
$V_1$                      & $V_2$                     &$V_3$                      & $V_4$\\
\midrule
$x_1$                      & $1/2-x_1$         & 0                         & 0 \\
$x_1$                      & 0                         & $1/2-x_1$        & 0 \\
$1/2-x_2-x_3$    & $x_2$                     & $x_3$                     & 0 \\
$x_1$                      & 0                         & 0                         & $1/2-x_1$\\
$1/2-x_2-x_4$    & $x_2$                     & 0                         & $x_4$ \\
$1/2-x_3-x_4$    &  0                        & $x_3$                     & $x_4$ \\
\bottomrule
\end{tabular}
\caption{\label{table.label} different schemes for choosing $n/2$ vertices from $G$.}
\end{table}

We choose $n/2$ vertices according to Table 3. Similar as above, we can obtain that
\begin{align}
& \frac{1/2-x_1}{x_i}e_{1i} \ge \frac{n^2}{18} \mbox{ ~for~ } i \in \{2,3\},  \tag{3.1} \\
& e_{23} + \frac{1/2-x_2-x_3}{x_1}(e_{12}+e_{13}) \ge \frac{n^2}{18}, \tag{3.2} \\
& \frac{1/2-x_1}{x_4}e_{14} + \left(\frac{1/2-x_1}{x_4}\right)^2 e_4 \ge \frac{n^2}{18}, \mbox{ ~and~ } \tag{3.3} \\
& e_{j4} + e_4 + \frac{1/2-x_j-x_4}{x_1}(e_{1j}+e_{14}) \ge \frac{n^2}{18} \mbox{ ~for~ } j\in \{2,3\}. \tag{3.4}
\end{align}

By simplifying the linear combination of
\begin{align}
\sum_{i=2,3}\left(\frac{x_i^2}{(1/2-x_1)x_1} \times (3.1)\right)
+ (3.2) + \frac{x_4^2}{(1/2-x_1)x_1}\times (3.3) + \sum_{j=2,3}(3.4), \notag
\end{align}
we derive that
\begin{align}
e(G) + \frac{e_4}{2x_{1}}
& \ge \left(\frac{1}{6}+\frac{x^2_2+x^2_3+x^2_4}{9x_1(1-2x_1)}\right)n^2. \notag
\end{align}
Since $e_4 \le |V_{4}|^2/3 = x_{4}^2n^2/3$, the inequality above implies that
\begin{align}
c + \frac{x_4^2}{6x_1} \ge \frac{1}{6}+\frac{x^2_2+x^2_3+x^2_4}{9x_1(1-2x_1)}. \notag
\end{align}
This is a contradiction due to the following claim whose proof can be found in the appendix.
\begin{claim}\label{claim-case3}
Under the conditions of Case 3, we have
\begin{align}
\frac{1}{6}+ \frac{x^2_2+x^2_3+x^2_4}{9x_1(1-2x_1)} - \frac{x_4^2}{6x_1}  - c > 0. \notag
\end{align}
\end{claim}
This contradiction completes the proof of Case 3.

\bigskip

\textbf{Case 4:} $x_2 + x_3 \le x_1 + x_3 < \frac{1}{2} \le x_1 + x_2$.

\begin{table}[htbp]
\centering
\begin{tabular}{llll}
\toprule
$V_1$                      & $V_2$                     &$V_3$                      & $V_4$\\
\midrule
$x_1$                      & $1/2-x_1$         & 0                         & 0 \\
$x_1$                      & 0                         & $1/2-x_1$             & 0 \\
$x_1$                      &$1/2-x_1-x_3$    & $x_3$                     & 0 \\
$1/2-x_2-x_3$    & $x_2$                     & $x_3$                     & 0 \\
$x_1$                      & 0                         & 0                         & $1/2-x_1$\\
0                          & $x_2$                     & 0                         & $1/2-x_2$\\
$1/2-x_3-x_4$    &  0                        & $x_3$                     & $x_4$ \\
0    &  $1/2-x_3-x_4$                        & $x_3$                     & $x_4$ \\
\bottomrule
\end{tabular}
\caption{\label{table.label} different schemes for choosing $n/2$ vertices from $G$.}
\end{table}

Choosing $n/2$ vertices according to Table 4, we obtain that
\begin{align}
& \frac{1/2-x_i}{x_{3-i}}e_{i,3-i} \ge \frac{n^2}{18} \quad \Rightarrow \quad  e_{i,3-i} \ge \frac{x_{3-i}}{1/2-x_i} \frac{n^2}{18} \mbox{ ~~for each~} i \in [2], \tag{4.1} \\
& e_{j3}+\frac{1/2-x_j-x_3}{x_2}(e_{12}+e_{j3}) \ge \frac{n^2}{18} \mbox{ ~~for each~} j\in [2], \tag{4.2} \\
& \frac{1/2-x_k}{x_4}e_{k4}+\left(\frac{1/2-x_k}{x_4}\right)^2e_4 \ge \frac{n^2}{18} \mbox{ ~~for each~} k\in [2], \mbox{ ~~and }\tag{4.3} \\
& e_{34}+e_4+\frac{1/2-x_3-x_4}{x_{\ell}}(e_{\ell3}+e_{\ell4}) \ge \frac{n^2}{18} \mbox{ ~~for each~} \ell \in [2]. \tag{4.4}
\end{align}
By simplifying the linear combination of
\begin{align}
&\frac{1}{2}\left(1+\frac{1}{1-2x_3}-\frac{1}{x_1+x_2}\right)\sum_{i\in[2]}(4.1)
+  \frac{1}{1-2x_3}\sum_{j\in[2]}\left(\frac{x_{3-j}}{x_1+x_2} \times (4.2)\right) \notag\\
+ & \frac{1}{2(x_1+x_2)}\sum_{k\in[2]}\left(\frac{x_4}{1/2-x_k} \times (4.3) \right)
+  \frac{x_1+x_2}{1/2-x_3-x_4} \sum_{\ell\in[2]}\left(\frac{x_{\ell}}{1/2-x_3-x_4} \times (4.4) \right), \notag
\end{align}
it yields that
\begin{align}
e(G) + \frac{1-x_1-x_2}{2(x_1+x_2)x_4}e_4
& \ge \left( \frac{1}{2}\left(1+\frac{1}{1-2x_3}-\frac{1}{x_1+x_2}\right)\left(\frac{x_2}{1/2-x_1}+\frac{x_1}{1/2-x2}\right) \right.\notag\\
& \quad \left.+ \frac{1}{1-2x_3} + \frac{1}{2(x_1+x_2)}\left(\frac{x_4}{1/2-x_1}  + \frac{x_4}{1/2-x_2} \right) +1 \right) \frac{n^2}{18}. \notag
\end{align}
Since $e_4 \le |V_{4}|^2/3 = x_{4}^2n^2/3$, the inequality above implies that
\begin{align}
c + \frac{(1-x_1-x_2)x_4}{6(x_1+x_2)}
& \ge \frac{1}{18} \left(\frac{1}{2} \left(1+\frac{1}{1-2x_3}-\frac{1}{x_1+x_2}\right)\left(\frac{x_2}{1/2-x_1}+\frac{x_1}{1/2-x2}\right) \right.\notag\\
& \quad \left.+ \frac{1}{1-2x_3} + \frac{1}{2(x_1+x_2)}\left(\frac{x_4}{1/2-x_1}  + \frac{x_4}{1/2-x_2} \right) +1 \right). \notag
\end{align}
Again, this is a contradiction because of the following claim, whose proof is included in the appendix.

\begin{claim}\label{claim-case4}
Under the conditions of Case 4, we have
\begin{align}
c + \frac{(1-x_1-x_2)x_4}{6(x_1+x_2)}
& < \frac{1}{18} \left( \frac{1}{2}\left(1+\frac{1}{1-2x_3}-\frac{1}{x_1+x_2}\right)\left(\frac{x_2}{1/2-x_1}+\frac{x_1}{1/2-x2}\right) \right.\notag\\
& \quad \left.+ \frac{1}{1-2x_3} + \frac{1}{2(x_1+x_2)}\left(\frac{x_4}{1/2-x_1}  + \frac{x_4}{1/2-x_2} \right) +1 \right). \notag
\end{align}
\end{claim}

This completes the proof of Theorem \ref{thm-case-medium-K4}.
\end{proof}

\subsection{Proof of Theorem \ref{thm-sparse-half-regular-K4}}\label{subsec:together}
Let $G$ be a $K_4$-free regular graph on $n$ vertices such that every vertex set of size $\lf n/2 \rf$ in $G$ spans at least $n^2/18$ edges.
Our goal is to show that $n$ is divisible by 6 and $G\cong T_{3}(n)$

First we show that it suffices to consider the case that $n$ is divisible by $6$.
Assume that we have proved for all $n$ that are divisible by $6$,
and now consider the case that $n$ is not divisible by $6$.
Let $H$ be the blow-up of $G$ obtained by replacing every vertex $i\in V(G)$ by a set $V_i$ of size $6$
and replacing every edge $ij\in E(G)$ by a complete bipartite graph with parts $V_i$ and $V_j$.
Then $H$ contains $N:= 6n$ vertices and is $K_4$-free and regular,
hence by our assumption, if let $S\subset V(H)$ be a subset of size $N/2 = 3n$ spanning the minimum number of edges,
then we have $e(S)\le N^2/18$.
We may assume that $S$ either contains $V_i$ or is disjoint from $V_i$ for all but at most one $i$,
since if there are two indices $i,j$ satisfying $1\leq |S\cap V_\ell|\leq 5$ for $\ell\in \{i,j\}$,
then we could increase one of the intersections and decreasing the other until $|S\cap V_\ell|\in \{0,6\}$ for some $\ell$, without increasing $e(S)$.
So, $S$ contains $\lf n/2 \rf$ sets $V_i$.
By our assumption on $G$, $e(S)\geq 36 \lc n^2/18 \rc > N^2/18$, a contradiction.

Now we assume that $n$ is divisible by $6$.
Let $e(G)=cn^2$ for some $c\in (0,1/3]$. Then every vertex in $G$ has degree $2cn$.
If $c\leq 0.26$, then by Theorem~\ref{thm-case-sparse-K4},
there exists a vertex set of size $n/2$ that spans strictly less than $n^2/18$ edges, a contradiction.
If $c\geq 0.295$, then by Theorem~\ref{thm-case-dense-K4}, we can derive that $G\cong T_{3}(n)$.
So it remains to consider $0.26<c<0.295$.
In this case, by Theorem~\ref{thm-case-medium-K4}, $G$ contains a vertex set of size $n/2$
that spans strictly less than $n^2/18$ edges, again a contradiction.
We have completed the proof of Theorem \ref{thm-sparse-half-regular-K4}. \qed

\section{Concluding remarks}
In this paper we consider the local density problem,
and prove Conjecture \ref{conj-Chung-Graham-sparse-half-K4} for all $K_4$-free regular graphs.
To fully resolve Conjecture \ref{conj-Chung-Graham-sparse-half-K4} there are two barriers in our proofs.
The first one is the minimum degree condition:
the proof of Theorem \ref{thm-case-dense-K4} requires the minimum degree to be at least $4n/7$ for the structure from Theorem~\ref{thm-Lyle-structure-K4},
while the proof of Theorem \ref{thm-case-medium-K4} requires the minimum degree to be at least $n/2$ for a good lower bound on the number of triangles.
The other barrier is the regular condition:
the proof of Case 1 of Theorem \ref{thm-case-medium-K4} (i.e., the footnote 3) requires $G$ to be regular for obtaining a lower bound on the number
of edges that contains at least one vertex of $V_4$.

A closely related problem is the problem of making a graph bipartite.
A famous conjecture of Erd\H{o}s \cite{E76} states that every triangle-free graph on $n$ vertices can be made bipartite by deleting at most $n^2/25$ edges.
This is still open, with the extremal graphs to be the balanced blow-ups of the 5-cycle.
Following from Krivelevich's observation \cite{K95}, we see that for regular graphs,
Conjecture~\ref{conj-erdos-sparse-halves-triangle} would imply the above conjecture of Erd\H{o}s.
So it seems interesting (but perhaps still difficult) to attack Conjecture~\ref{conj-erdos-sparse-halves-triangle} for regular graphs.

For analogous problems on $K_r$-free graphs and other related problems, we direct interested readers to \cite{CG90,EFPS88,EFRS94,S07}.

\section*{Acknowledgement}
The authors would like to thank Jun Gao for carefully reading a draft of this paper.

\bibliographystyle{abbrv}
\bibliography{local_density}

\section*{Appendices}

\appendix

\section{Proof of Theorem \ref{obsevation-triangle-free-Kriveleich}}
In this section we prove Theorem \ref{obsevation-triangle-free-Kriveleich}.
We need the following lemmas.

\begin{lemma}[\cite{EFRS94}]\label{lemma-triangle-large-matching}
Let $0 < \alpha \le 1$ and let
$G$ be a triangle-free graph on $\alpha n$ vertices with at least $(2\alpha-1)n^2/4$ edges.
Then $G$ contains a matching with at least $(2\alpha -1)n/2$ edges.
\end{lemma}

\begin{lemma}\label{lemma-bipartite-local-density}
Let $1/2 \le \alpha \le 1$, $n \in \mathbb{N}$ and $\alpha n\in\mathbb{N}$.
Let $G$ be bipartite graph on $n$ vertices.
If every vertex set of size $\alpha n$ in $G$ spans at least $(2\alpha-1)n^2/4$ edges,
then $G \cong T_{2}(n)$.
\end{lemma}
\begin{proof}
Let $V_1\cup V_2 = V(G)$ be a partition such that $G$ is a bipartite graph with parts $V_1$ and $V_2$.
Let $x = |V_1|/n$ and we may assume that $x\ge 1/2$.
By assumption, $x < \alpha$, since otherwise there would be a subset of $A$ of size $\alpha n$
that spans zero edges.
Now choose a random set $S\subset B$ with $|S| = (\alpha-x)n$.
Then $|A\cup B| = \alpha n$ and
$e(A \cup S) \le x (\alpha-x)n^2 \le (2\alpha-1)n^2/4$.
By assumption the inequality above must be tight, which means
$x=1/2$ and $G[A,S]$ is a complete bipartite graph.
Since $S$ was chosen randomly, $G$ must be a complete bipartite graph with $|V_1|=n/2$.
Therefore, $G \cong T_2(n)$.
\end{proof}


Now we prove Theorem \ref{obsevation-triangle-free-Kriveleich}.

\begin{proof}[Proof of Theorem \ref{obsevation-triangle-free-Kriveleich}]
First one could see from Kriveleich's proof (i.e. the proof of Theorem 4) in \cite{K95}
that if $G$ does not contain an independent set of size $(1-\alpha)n$, then
there exists a vertex set of size $\alpha n$ in $G$ that spans strictly less than $\frac{2\alpha -1}{4}n^2$ edges.
So by assumption there exists an independent set in $G$ whose size is $(1-\alpha)n$.
Next, we use the argument of Erd\H{o}s et al. \cite{EFRS94} to show that $G \cong T_2(n)$.

Let $A \subset V(G)$ be an independent set of size $(1-\alpha)n$.
By Lemma \ref{lemma-triangle-large-matching},
there exists a matching $M$ in $G[V(G)\setminus A]$ with $(2\alpha -1)n/2$ edges.
Let $C = V(M)$ and let $B = V(G)\setminus (A\cup C)$.
Note that $|C| = (2\alpha -1)n$.

Since $G$ is triangle-free and $M$ is a matching,
every vertex in $A$ is adjacent to at most half of the vertices in $C$.
Therefore, $e(A,C) \le (1-\alpha)(2\alpha-1)n^2/2$ and hence
\begin{align}
e(A\cup C)
= e(A,C) + e(C)
\le \frac{ (1-\alpha)(2\alpha-1)n^2}{2} + \frac{(2\alpha -1)^2n^2}{4}
= \frac{2\alpha-1}{4}n^2. \notag
\end{align}
Since $|A\cup C| = \alpha n$, by assumption, $e(A\cup C) \ge (2\alpha-1)n^2/4$.
So all inequalities above must be tight, which means
$G[C]$ is a balanced complete bipartite graph.
and every vertex in $A$ is adjacent to exactly half of the vertices in $C$.

Let $C_1 \cup C_2 = C$ be a partition such that $G[C] = G[C_1,C_2]$ and note that
$|C_1| = |C_2| = |C|/2 = (2\alpha-1)n/2$.
For $i \in \{1,2\}$ let
\begin{align}
    A_i  = \{u\in A: \exists v\in C_i, uv\in E(G)\}
    \quad {\rm and} \quad
    B_i  = \{u\in B: \exists v\in C_i, uv\in E(G)\}, \notag
\end{align}
and let $B_3 = B\setminus(B_1\cup B_2)$.
Since $G[C_1,C_2]$ is a complete bipartite graph and every $v \in A$ is adjacent to at least half vertices in $C_1\cup C_2$,
we have $uw\in E(G)$ for all $u \in A_i$ and $w \in C_{i}$ for $i\in\{1,2\}$.
Notice that $A_1 \cup A_2$ is a partition of $A$, and
for $i\in \{1,2\}$ we have $uv\not\in E(G)$ for all $u\in B_i, v\in A_{i}$,
since otherwise there exists $w\in C_1$ such that $u,v,w$ induces a copy of $K_3$ in $G$, a contradiction.
Therefore, if $B_3 = \emptyset$, then $G$ is bipartite with two parts $V_1 = C_1 \cup A_2 \cup B_2$ and
$V_2 = C_2 \cup A_1\cup B_1$, and by Lemma \ref{lemma-bipartite-local-density}, $G \cong T_2(n)$.
So we may assume that $B_3 \neq \emptyset$.

Let $\widehat{C}_{1} = C_1 \cup B_2$ and $\widehat{C}_{2} = C_2 \cup B_1$.
Let $x_i = |A_i|/n$, $y_i = |\widehat{C}_{i}|/n$ for $i\in\{1,2\}$, and $z = |B_3|/n$.
Since $|\widehat{C}_{1} \cup \widehat{C}_{2} \cup A_1 \cup B_3| = n-|A_{2}| \ge \alpha n$ ,
there exists $U_1 \subset \widehat{C}_{1}$ with
$|U_1| = \alpha n - |B_3 \cup A_1 \cup \widehat{C}_{2}| = (\alpha - z- x_1 -y_{2})n$.
Since $|B_3 \cup A_1 \cup \widehat{C}_{2} \cup U_1| = \alpha n$, by assumption
\begin{align}
\frac{2\alpha-1}{4}n^2 \le e(B_3 \cup A_1 \cup \widehat{C}_{2} \cup U_1)
\le z x_1n^2 + (x_1+y_2)(\alpha - z- x_1 -y_{2})n^2. \notag
\end{align}
Similarly, there exists $U_2 \subset  \widehat{C}_{2}$ with
$|U_2| = (\alpha - z- x_2 -y_{2})n$, and
\begin{align}
\frac{2\alpha-1}{4}n^2 \le e(B_3 \cup A_2 \cup \widehat{C}_{1} \cup U_2)
\le z x_2 n^2 + (x_2+y_1)(\alpha - z- x_2 -y_{1})n^2. \notag
\end{align}
Adding up these two inequalities we obtain
\begin{align}
\frac{2\alpha-1}{2}
& \le z x_1 + (x_1+y_2)(\alpha - z- x_1 -y_{2})
 + z x_2  + (x_2+y_1)(\alpha - z- x_2 -y_{1}) \notag\\
& = \alpha(x_1+x_2+y_1+y_2)-z(y_1+y_2) -\left( (x_1+y_2)^2+(x_2+y_1)^2\right) \notag\\
& \le \alpha (1-z) -z (\alpha-z) - \frac{(x_1+x_2+y_1+y_2)^2}{2} \notag\\
& = \alpha (1-z) -z (\alpha-z) - \frac{(1-z)^2}{2}
 = \frac{z^2}{2} - (2a-1)z + \frac{2\alpha-1}{2}, \notag
\end{align}
which implies that $z^2/2 - (2a-1)z \ge 0$.
However, since $0 < z \le 1-\alpha < 4\alpha-2$ (here we used $\alpha > 3/5$ and $B_3\neq \emptyset$),
\begin{align}
\frac{z^2}{2} - (2a-1)z = \frac{z}{2}(z-(4\alpha-2)) < 0, \notag
\end{align}
a contradiction.
\end{proof}

\section{Proofs of Claims \ref{claim-case3} and \ref{claim-case4}}
In this section we prove Claims \ref{claim-case3} and \ref{claim-case4}.

\begin{proof}[Proof of Claim \ref{claim-case3}]
Since $x_2^2 + x_3^2 \ge (x_2+x_3)^2/2$, it suffices to show
\begin{align}
\frac{(x_2+x_3)^2/2+x^2_4}{9x_1(1-2x_1)} - \frac{x_4^2}{6x_1} + \frac{1}{6} - c > 0. \notag
\end{align}
Plugging $x_4 = 1-g(c)$ and $x_2+x_3=g(c)-x_1$ into the inequality above, it becomes
\begin{align}
\ell(c,x_1) :=\frac{(g(c)-x_1)^2+2(1-g(c))^2}{18x_1(1-2x_1)} - \frac{(1-g(c))^2}{6x_1} + \frac{1}{6} - c > 0. \tag{3.5}
\end{align}
Then with the aid of computer \cite{CAL}
one can see that
\begin{align}
\min\left\{\ell(c,x): x\in (0,1/2), c\in [1/4,1/3]\right\} > 0.003. \notag
\end{align}
Therefore, $(3.5)$ is true.
\end{proof}

\begin{proof}[Proof of Claim \ref{claim-case4}]
First since $1/(1/2-x)$ is concave up, by Jensen's inequality
\begin{align}
\frac{x_4}{1/2-x_1}  + \frac{x_4}{1/2-x_2} \ge \frac{4x_4}{1-(x_1+x_2)}. \notag
\end{align}
Since $x_1+x_2 \ge 1/2$ and $x_2\le x_1 < 1/2$,
\begin{align}
\frac{x_2}{1/2-x_1}+\frac{x_1}{1/2-x_2} - \frac{2(x_1+x_2)}{1-(x_1+x_2)}
= \frac{2(x_1-x_2)^2(2x_1+2x_2-1)}{(1-2x_1)(1-2x_2)(1-x_1-x_2)} \ge 0. \notag
\end{align}
It suffice to show that
\begin{align}
c + \frac{(1-x_1-x_2)x_4}{6(x_1+x_2)}
& < \frac{1}{18} \left( \left(1+\frac{1}{1-2x_3}-\frac{1}{x_1+x_2}\right) \cdot \frac{x_1+x_2}{1-(x_1+x_2)} \right.\notag\\
& \quad \left.+ \frac{1}{1-2x_3} + \frac{1}{2(x_1+x_2)}\cdot \frac{4x_4}{1-(x_1+x_2)} +1 \right), \notag
\end{align}
Let $x = x_1+x_2$. Then $x_3 = g(c)-x$ and the inequality above can be simplified as
\begin{align}
m(x,c) & := \left(1+\frac{1}{1-2(g(c)-x)}-\frac{1}{x}\right) \cdot \frac{x}{1-x}  \notag\\
& \quad +  \frac{1}{1-2(g(c)-x)} + \frac{2(1-g(c))}{x(1-x)} +1 -  \frac{3(1-x)(1-g(c))}{x} -18c>0. \tag{4.5}
\end{align}
Then with the aid of computer \cite{CAL}
one can see that
\begin{align}
\min\left\{m(x,c): x\in [1/2,1], c\in [1/4,1/3]\right\} > 0.099. \notag
\end{align}
Therefore, $(4.5)$ is true.
\end{proof}

\end{document}